\documentclass[11pt,leqno,oneside,letterpaper]{amsart}

\usepackage[T1]{fontenc}
\usepackage[utf8]{inputenc}
\usepackage{lmodern}
\usepackage{eulervm}
\usepackage{amsmath,amssymb,amsthm,mathtools}
\usepackage[margin=1.15in]{geometry}
\usepackage{booktabs,array}
\usepackage{enumitem}
\usepackage{xurl}
\usepackage[colorlinks,linkcolor=blue,citecolor=blue]{hyperref}
\usepackage[nameinlink,noabbrev]{cleveref}
\usepackage{orcidlink}
\hypersetup{pdftitle={Prescribed Abscissae on Congruent-Number Curves over Simplest Cubic Fields},pdfauthor={Junyu Lu},pdfsubject={Prescribed coordinates and Galois trace on elliptic curves},pdfkeywords={elliptic curve, simplest cubic field, congruent number, Galois trace, Pell equation}}

\newtheorem{theorem}{Theorem}
\newtheorem{proposition}[theorem]{Proposition}

\newtheorem{lemma}[theorem]{Lemma}
\theoremstyle{remark}

\newtheorem{example}[theorem]{Example}

\newcommand{\QQ}{\mathbb Q}
\newcommand{\ZZ}{\mathbb Z}

\newcommand{\FF}{\mathbb F}

\newcommand{\Norm}{\operatorname{N}}
\newcommand{\Tr}{\operatorname{Tr}}
\newcommand{\rank}{\operatorname{rank}}
\newcommand{\tors}{\mathrm{tors}}

\title[Prescribed abscissae over cubic fields]{Prescribed Abscissae on Congruent-Number Curves\\over Simplest Cubic Fields}
\author[J.~Lu]{Junyu Lu\,\orcidlink{0009-0001-4973-3198}}
\address{Department of Mathematics, Yang-En University, Majia Town, Luojiang District, Quanzhou City, Fujian Province, China, 362014}
\email{jy.lu@outlook.com}
\date{}

\begin{document}
\begin{abstract}
      For integers $t\geq -1$, let $\theta_t$ be the largest real root of the Shanks polynomial $g_t(X)=X^3-tX^2-(t+3)X-1$ and put $K_t=\QQ(\theta_t)$. We classify the points on $E_n:y^2=x^3-n^2x$ over $K_t$ with abscissa $\theta_t-1$ when $n$ is a positive integer and $E_n(\QQ)$ has rank zero. We determine the four pairs $(d,t)$ with $d>0$ rational for which $(\theta_t-1)/d^2$ is an abscissa on $E_3(K_t)$. We also exclude the abscissa $\theta_t-1$ on $E_5(K_t)$ and $E_6(K_t)$ and prove that only finitely many parameters $t$ admit this abscissa for each fixed positive integer $n$.
\end{abstract}
\maketitle

\section{Introduction}

For an integer $t\geq-1$, consider the Shanks polynomial
\[
      g_t(X)=X^3-tX^2-(t+3)X-1
\]
and let $\theta_t$ be its largest real root. Set
\[
      K_t=\QQ(\theta_t),\qquad \Delta_t=t^2+3t+9.
\]
By the rational root theorem, the only possible rational roots are $1$ and $-1$. Since $g_t(-1)=1$ and $g_t(1)=-2t-3\ne0$, the polynomial is irreducible over $\QQ$. The discriminant formula $\operatorname{disc}(g_t)=\Delta_t^2>0$ shows that it has three distinct real roots. An irreducible cubic with square discriminant has cyclic Galois group of order three, so $K_t/\QQ$ is a totally real cyclic cubic extension. Here $\Delta_t^2$ is the discriminant of the polynomial $g_t$, or equivalently of the order $\ZZ[\theta_t]$; this order need not be the full ring of integers of $K_t$. Shanks's study \cite{Shanks1974} of this family includes its Galois action, units, regulators, and class numbers. The identity
\[
      g_{-t-3}(-X-1)=-g_t(X)
\]
explains the usual restriction $t\geq-1$. Distinct parameters in this range can still define the same field. Throughout this article, $t$ is an integer with $t\geq -1$ unless otherwise specified.

For a positive integer $n$, let
\[
      E_n:\quad y^2=x^3-n^2x.
\]
We study points on $E_n$ over simplest cubic fields $K_t$ with a prescribed abscissa.

The conjugation formulas give the three roots of $g_t(X)$ as
\[
      \theta_t,\qquad -\frac1{\theta_t+1},\qquad
      -\frac{\theta_t+1}{\theta_t}.
\]
We choose the generator $\sigma$ with $\sigma(\theta_t)=-1/(\theta_t+1)$. Since the group law on $E_n$ is defined over $\QQ$, the sum of the three conjugates of a point is fixed by $\sigma$. For $P\in E_n(K_t)$, the trace map $\Tr:E_n(K_t)\to E_n(\QQ)$ is given by
\[
      \Tr(P)=P+\sigma P+\sigma^2P\in E_n(\QQ).
\]
The sum here uses the elliptic-curve group law. For elements of $K_t$, we write $\Tr_{K_t/\QQ}$ and $\Norm_{K_t/\QQ}$ for the field trace and norm, respectively; in particular, $\Tr_{K_t/\QQ}(\alpha)=\alpha+\sigma(\alpha)+\sigma^2(\alpha)$ uses addition in the field.

The identity of the elliptic curve is denoted by $O$. The condition $x(P)=\theta_t-1$ depends on the chosen root and parameter, not just on the isomorphism class of $K_t$. Our first result determines all points with this abscissa when the rational rank is zero.

\begin{theorem}\label{thm:main-classification}
      Suppose that $t\geq -1, n>0$ are integers and $\rank E_n(\QQ)=0$. There exists a point in $E_n(K_t)$ with abscissa $\theta_t-1$ if and only if there is a positive integer $v$ such that
      \[
            v^2=2n^2-9,\qquad
            t=(n+v)^2+3\quad\text{or}\quad t=(n-v)^2+3.
      \]
      At $t=(n+v)^2+3$, the ordinates are exactly $\pm((n+v)\theta_t+n)$; at $t=(n-v)^2+3$, they are exactly $\pm((n-v)\theta_t+n)$. Every such point has trace zero and infinite order, and its conjugates generate a subgroup of rank two.
\end{theorem}

The main step is to exclude nonzero two-torsion trace. The exclusion holds even when $n$ is positive rational and without a rank assumption. When the rational rank is zero, every rational point is two-torsion, so the exclusion forces the trace to vanish. The conjugates then lie on a rational line, and coefficient comparison gives the Pell equation above.

We next fix $E_3$ and allow a rational square in the denominator of the abscissa. Using the same equation to define $E_n$ also for positive rational $n$, scaling gives the abscissa $\theta_t-1$ on another curve in the family.

\begin{theorem}\label{thm:main-fixed}
      Suppose that $d$ is a positive rational number and $t\geq-1$ is an integer. There exists a point in $E_3(K_t)$ with abscissa $(\theta_t-1)/d^2$ if and only if
      \[
            (d,t)\in\{(1,3),(1,39),(13,44103),(13,1498179)\}.
      \]
\end{theorem}

Scaling to $E_{3d^2}$ reduces this assertion to the classical equation $B^2+1=2d^4$. Every admissible rational scaling factor is integral. We give the ordinates and distinguish the four fields directly.

The rank-zero hypothesis in Theorem~\ref{thm:main-classification} forces the rational trace to be torsion. Without this hypothesis, a nontorsion trace is not ruled out. Nevertheless, any point with the prescribed abscissa gives $\rank E_n(K_t)\geq\rank E_n(\QQ)+2$, and its conjugates generate a subgroup of rank two or three according as its trace is zero or nonzero (Proposition~\ref{prop:rank}).

\begin{theorem}\label{thm:main-positive-rank}
      For every integer $t\geq-1$, neither $E_5(K_t)$ nor $E_6(K_t)$ contains a point with abscissa $\theta_t-1$. For every fixed positive integer $n$, only finitely many integers $t\geq-1$ admit a point with this abscissa on $E_n(K_t)$.
\end{theorem}

The $E_5$ assertion follows from square classes and a local obstruction at $5$; the $E_6$ assertion follows from the norm equation alone. For finiteness, Siegel's theorem \cite[IX.3, Corollary~3.2.2]{Silverman2009} applies to the same norm equation. These assertions hold without assuming that the rational rank is zero.

Rational lines through three conjugate points are a standard construction in cyclic cubic rank growth; Kisilevsky's account \cite[Section~3]{Kisilevsky2012} explains this geometry. For an elliptic curve $E/F$ over a number field $F$, Fearnley, Kisilevsky, and Kuwata \cite[Theorem~5.1]{FearnleyEtAl2012} prove rank growth in infinitely many cyclic cubic extensions of $F$ when $E$ has a nontorsion point over a cyclic extension of $F$ of degree dividing three. Girondo et al.'s construction \cite[Theorem~3]{GirondoEtAl2009}, applied after reducing the curve parameter to its squarefree part and then scaling back, gives nontorsion points on congruent-number curves over explicit cubic fields with exactly one real embedding. The construction at the beginning of Joshi's Section~3.8 \cite{Joshi2018}, with $a=d=3$ and the change of coordinate $x=X-1$, gives $(\theta_3-1,3)$ on $E_3$.

When the rational rank is zero, our results classify points with the prescribed coordinate in the cyclic Shanks family. The trace obstruction and the separate exclusions for $E_5$ and $E_6$ are the main arguments. A final extension allows the shift in the coordinate to vary. The resulting infinite family on $E_3$ explains the role of the fixed shift in Theorem~\ref{thm:main-fixed}.

Section~\ref{sec:marked-classification} proves Theorem~\ref{thm:main-classification}, and Section~\ref{sec:scaled} treats rational scaling. Section~\ref{sec:pell-family} gives the Pell family as $n$ varies. Section~\ref{sec:quadratic-ordinate} proves Theorem~\ref{thm:main-positive-rank}, gives a quartic criterion for the existence of an ordinate at each fixed pair $(n,t)$, and describes the possible nontorsion traces. Section~\ref{sec:shifted} treats integral shifts and constructs the resulting family on $E_3$.

\section{Trace and the classification}\label{sec:marked-classification}

We first describe points with abscissa $\theta_t-1$ and trace zero by a rational line, then exclude the other possible traces when the rational rank is zero. We allow $n$ to be positive rational for the application in Section~\ref{sec:scaled}, where a rational change of coordinates replaces $n$ by $3d^2$.

Set $z=\theta_t-1$ and $f_n(X)=X^3-n^2X$. The minimal polynomial of $z$ is
\begin{equation}\label{eq:q}
      q_t(Z)=g_t(Z+1)=Z^3+(3-t)Z^2-3tZ-(2t+3).
\end{equation}
Thus both $1,z,z^2$ and $1,\theta_t,\theta_t^2$ are $\QQ$-bases of $K_t$. Writing the ordinate in either basis turns the point equation into a polynomial identity.

J\k{e}drzejak's torsion theorem \cite[Lemma~14]{Jedrzejak2012} gives, for a number field $F\subset\mathbb R$ with $[F:\QQ]$ odd,
\begin{equation}\label{eq:torsion}
      E_n(F)_{\tors}=E_n[2]=\{O,(0,0),(n,0),(-n,0)\}.
\end{equation}
Here $E_n[2]$ denotes the subgroup killed by multiplication by two. The theorem's exclusions $\sqrt2,\sqrt3,\sqrt5\notin F$ follow from the tower law, since a field of odd degree cannot contain a quadratic subfield. The fields $K_t$ used here are in fact totally real.

The cited statement assumes that $n$ is squarefree, but it also implies \eqref{eq:torsion} for every positive rational $n$. Write $n=n_0s^2$, where $n_0$ is a positive squarefree integer and $s\in\QQ_{>0}$, and apply the rational isomorphism $E_n\to E_{n_0}$ given by $(x,y)\mapsto(x/s^2,y/s^3)$.

The chord-and-tangent law \cite[III.2, Composition Law~2.1 and Proposition~2.2(a)]{Silverman2009} relates trace zero to collinearity. The distinctness assumption in the next lemma excludes vertical lines and repeated intersections; it holds when the abscissa is $\theta_t-1$, which has degree three.

\begin{lemma}\label{lem:line}
      If $P\in E_n(K_t)$ has three distinct conjugate abscissae, then $\Tr(P)=O$ if and only if its conjugates lie on a nonvertical line defined over $\QQ$.
\end{lemma}
\nopagebreak
\begin{proof}
      If $\Tr(P)=O$, the chord-and-tangent law places $P,\sigma P,\sigma^2P$ on a line. Their abscissae are distinct, so the line is nonvertical and has an equation $y=\ell_1x+\ell_0$ over $K_t$. Applying $\sigma$ gives another line through the same three points. Uniqueness of the line through two distinct points gives $\sigma(\ell_1)=\ell_1$ and $\sigma(\ell_0)=\ell_0$, so $\ell_1,\ell_0\in\QQ$. Conversely, three intersections with a nonvertical line have group sum $O$, which gives $\Tr(P)=O$.
\end{proof}

For the abscissa $\theta_t-1$, the line equation expresses the ordinate as a polynomial of degree at most one in $\theta_t$. Substituting this expression into the curve equation gives a cubic whose coefficients can be compared with those of $g_t$.

\begin{proposition}\label{thm:trace-pell}
      Let $n$ be a positive rational number and  $P=(\theta_t-1,y)\in E_n(K_t)$. Then $\Tr(P)=O$ if and only if $n$ is an integer and there are $u\in\ZZ$ and $\epsilon\in\{1,-1\}$ such that
      \begin{equation}\label{eq:pell-data}
            y=u\theta_t+\epsilon n,\qquad
            t=u^2+3,\qquad (u-\epsilon n)^2-2n^2=-9.
      \end{equation}
\end{proposition}
\begin{proof}
      If $\Tr(P)=O$, Lemma~\ref{lem:line} gives $y=u\theta_t+w$ with $u,w\in\QQ$. Consider
      \[
            H(X)=(X-1)^3-n^2(X-1)-(uX+w)^2.
      \]
      Since $H(\theta_t)=0$, the minimal polynomial $g_t$ divides $H$. Both are monic cubics, so they are equal. Expanding $H$ gives
      \[
            H(X)=X^3-(3+u^2)X^2+(3-n^2-2uw)X+n^2-1-w^2.
      \]
      Comparison with $g_t$ now yields
      \[
            t=u^2+3,\qquad n^2+2uw=u^2+9,\qquad w^2=n^2.
      \]
      The last equation gives $w=\epsilon n$, and the middle equation becomes $(u-\epsilon n)^2-2n^2=-9$.

      It remains to check integrality. A rational number whose square is an integer is itself an integer, so $u^2=t-3$ gives $u\in\ZZ$. Substituting $w=\epsilon n$ into the middle equation makes $n$ a rational root of the monic integral polynomial $X^2+2\epsilon uX-u^2-9$. The rational root theorem then gives $n\in\ZZ$.

      Conversely, the conditions in \eqref{eq:pell-data} give
      \begin{align}
             & (X-1)^3-n^2(X-1)-(uX+\epsilon n)^2-g_{u^2+3}(X)\notag \\
             & \hspace{20mm}=\bigl((u-\epsilon n)^2-2n^2+9\bigr)X=0.
            \label{eq:identity}
      \end{align}
      The resulting point and its conjugates lie on $y=u(x+1)+\epsilon n$, so their trace is zero.
\end{proof}

This proposition describes exactly the points whose trace is zero, without any rank assumption. To obtain a classification when the rational rank is zero, we must also rule out a trace equal to one of the three nonzero points in $E_n[2]$. Translation by that point will reduce the question to the rational-line criterion.

\begin{theorem}\label{thm:obstruction}
      Suppose $n$ is a positive rational number and $t\geq-1$ is an integer. Then for any point $P=(\theta_t-1,y)\in E_n(K_t)$,
      \[
            \Tr(P)\notin E_n[2]\setminus\{O\}.
      \]
\end{theorem}
\begin{proof}
      Suppose $\Tr(P)=T_2=(r_2,0)\ne O$, where $r_2=0,n$, or $-n$, and put $Q=P+T_2$. The point $T_2$ is rational and has order two, so
      \[
            \Tr(Q)=\Tr(P)+3T_2=T_2+T_2=O.
      \]
      Since $z$ has degree three, $z\ne r_2$ and $Q$ is affine. Put $\mu=f_n'(r_2)=3r_2^2-n^2\ne0$. Substituting $y^2=f_n(z)$ into the addition formulas with slope $y/(z-r_2)$ gives
      \begin{equation}\label{eq:translate}
            x(Q)=r_2+\frac{\mu}{z-r_2},\qquad
            y(Q)=-\frac{\mu y}{(z-r_2)^2}.
      \end{equation}
      The fractional-linear expression for $x(Q)$ preserves distinctness of the conjugate abscissae: its denominators are nonzero and $\mu\ne0$. Lemma~\ref{lem:line} therefore gives $y(Q)=\ell_1 x(Q)+\ell_0$ with $\ell_1,\ell_0\in\QQ$. Substitution into \eqref{eq:translate} yields
      \[
            y=-\frac{\ell_1r_2+\ell_0}{\mu}(z-r_2)^2-\ell_1(z-r_2).
      \]
      If $\ell_1r_2+\ell_0=0$, the ordinate has degree at most one in $z$, which would force $\Tr(P)=O$. Hence we can write $y=J(z)$, where $J(Z)=j_2(Z-r_2)(Z+\beta)$ with $ j_2\in\QQ^\times$ and $\beta\in\QQ$.

      Since $J(z)^2=f_n(z)$, the minimal polynomial $q_t(Z)$ divides $J(Z)^2-f_n(Z)$. Also $J(r_2)=f_n(r_2)=0$, so $Z-r_2$ divides it. These factors are coprime because $q_t(Z)$ has no rational root. Comparison of degrees and leading coefficients therefore gives
      \[
            J(Z)^2-f_n(Z)=j_2^2(Z-r_2)q_t(Z).
      \]
      On putting $\lambda=j_2^{-2}>0$ and cancelling $j_2^2(Z-r_2)$, we obtain
      \[
            q_t(Z)=(Z-r_2)(Z+\beta)^2-\lambda(Z^2+r_2Z+r_2^2-n^2).
      \]
      Its coefficients give
      \begin{align}
            \lambda                        & =t-3+2\beta-r_2,\label{eq:lambda}  \\
            \beta^2-2r_2\beta-\lambda r_2  & =-3t,\label{eq:linear-coeff}       \\
            -r_2\beta^2-\lambda(r_2^2-n^2) & =-(2t+3).\label{eq:constant-coeff}
      \end{align}
      For $r_2=0$, the last equation is $\lambda n^2=-(2t+3)$; for $r_2=-n$, it is $n\beta^2=-(2t+3)$. Both are impossible because $t\geq-1$.

      For $r_2=n$, equation~\eqref{eq:constant-coeff} gives $t=(n\beta^2-3)/2$. Eliminating $\lambda$ from \eqref{eq:linear-coeff} first gives
      \[
            \beta^2-4n\beta+(3-n)t+3n+n^2=0.
      \]
      Substituting the value of $t$ and multiplying by two yields
      \begin{equation}\label{eq:B-quadratic}
            (2+3n-n^2)\beta^2-8n\beta+(2n^2+9n-9)=0.
      \end{equation}
      Its leading coefficient cannot vanish for rational $n$, since that would imply $(2n-3)^2=17$. Thus its discriminant must be a rational square. The discriminant is $4D(n)$, where
      \[
            D(n)=2n^4+3n^3-24n^2+9n+18.
      \]
      We show that $D(n)$ cannot be a rational square by reducing modulo $3$. Write $n=h/k$ with coprime positive integers $h,k$. Clearing the denominator gives
      \[
            k^4D(h/k)=2h^4+3h^3k-24h^2k^2+9hk^3+18k^4.
      \]
      If $3\nmid h$, this integer is $2$ modulo $3$. If $h=3\ell $, then $3\nmid k$, and its quotient by $9$ is
      \[
            18\ell ^4+9\ell ^3k-24\ell ^2k^2+3\ell k^3+2k^4\equiv2\pmod3.
      \]
      Multiplication by $k^4$ and division by $9$ preserve the property of being a rational square. In either case we obtain an integer congruent to $2$ modulo $3$. Since a rational square that is integral is an integer square, both cases are impossible. This contradiction excludes $r_2=n$.
\end{proof}

The Mordell--Weil theorem \cite[VIII.6, Theorem~6.7]{Silverman2009} supplies finite generation over number fields, which we use in the following rank calculation.

\begin{proposition}\label{prop:rank}
      Suppose $n$ is a positive rational number and $P=(\theta_t-1,y)\in E_n(K_t)$. Then $P$ has infinite order and
      \[
            \rank E_n(K_t)\geq\rank E_n(\QQ)+2.
      \]
      Its conjugates generate a subgroup of rank two if $\Tr(P)=O$, and of rank three otherwise.
\end{proposition}
\begin{proof}
      The abscissa $\theta_t-1$ has degree three, so it is not $0$ or $\pm n$. Equation~\eqref{eq:torsion} therefore excludes torsion. Put $T=\Tr(P)$ and $Q=3P-T$. Since $T$ is rational, $\Tr(Q)=3T-3T=O$. The kernel of the trace contains no nonzero torsion: all torsion points are rational and killed by two, so the trace acts on them as multiplication by three, which is the identity. If $Q=O$, applying $\sigma$ and subtracting would give $3(\sigma P-P)=O$. There is no nonzero three-torsion, so $\sigma P=P$, contrary to the degree of its abscissa. Thus $Q$ is nontorsion.

      Let $V=E_n(K_t)\otimes_{\ZZ}\QQ$. The trace extends $\QQ$-linearly to $V$. For any nonzero $v\in V$ with $\Tr(v)=0$, we have $\sigma^2v=-v-\sigma v$. A relation $av+b\sigma v=0$ with $a,b\in\QQ$, together with its conjugate, gives
      \[
            av+b\sigma v=0,\qquad -bv+(a-b)\sigma v=0.
      \]
      If $(a,b)\ne(0,0)$, the determinant $a^2-ab+b^2=(a-b/2)^2+3b^2/4$ is positive. The equations would then force $v=0$, a contradiction. Thus $v,\sigma v$ are independent, and the third conjugate lies in their span.

      Apply this to $v=Q\otimes1$. The trace vanishes on the resulting two-dimensional subspace and acts as multiplication by three on $E_n(\QQ)\otimes\QQ$. These subspaces therefore have zero intersection. Adding their dimensions proves the rank inequality.

      If $T=O$, the same argument applied to $P\otimes1$ gives rank two for the subgroup generated by the conjugates of $P$. If $T\ne O$, Theorem~\ref{thm:obstruction} and \eqref{eq:torsion} make $T$ nontorsion. The span of the conjugates of $P$ contains $T\otimes1$, $Q\otimes1$, and $\sigma Q\otimes1$. The last two are independent and have trace zero, whereas $\Tr(T\otimes1)=3T\otimes1\ne0$. The three vectors are independent, so the subgroup generated by the three conjugates of $P$ has rank three.
\end{proof}

\begin{proof}[Proof of Theorem~\ref{thm:main-classification}]
      The rank-zero hypothesis and \eqref{eq:torsion} imply $E_n(\QQ)=E_n[2]$. Since the trace is rational, Theorem~\ref{thm:obstruction} forces every point with abscissa $\theta_t-1$ to have trace zero. We may therefore apply Proposition~\ref{thm:trace-pell}.

      Negating a point changes $(u,\epsilon)$ to $(-u,-\epsilon)$, so we may take $\epsilon=1$ up to the sign of the ordinate. The equation becomes $(u-n)^2=2n^2-9$, giving $u=n\pm v$ and $t=(n\pm v)^2+3$. The case $v=0$ would require $2n^2=9$ and is impossible for integral $n$. Thus $v>0$, and the two defining parameters differ by $4nv>0$.

      Conversely, these values satisfy \eqref{eq:identity}, so each listed point lies on the curve and has trace zero. At this abscissa, the right-hand side of the curve equation is nonzero, so there are exactly two ordinates, differing by sign. This proves completeness of the list, and Proposition~\ref{prop:rank} gives infinite order and independence of two conjugates.
\end{proof}

In particular, when the rational rank is zero and $3\nmid n$, no point with abscissa $\theta_t-1$ exists, since $v^2=2n^2-9$ would give $v^2\equiv2\pmod3$.

\section{Rationally scaled coordinates}\label{sec:scaled}

Theorem~\ref{thm:main-classification} fixes the form of the coordinate and allows the curve to vary. A rational scaling lets us apply it to the fixed curve $E_3$ while varying the denominator of the coordinate. We first recall the rational-rank calculation needed for this application.

In this descent calculation, let $B\in\ZZ\setminus\{0\}$ and consider the curve $y^2=x^3+Bx$. The square-class map used in $2$-isogeny descent \cite[X.4, Proposition~4.9; X.6]{Silverman2009} sends a point with $x\ne0$ to the class of $x$ in $\QQ^\times/\QQ^{\times2}$. It sends $O$ to $1$ and $(0,0)$ to the class of $B$.

To identify the possible classes, take an affine rational point with $x\ne0$ and write $x=r/s^2$, $y=q/s^3$, with $r,s,q\in\ZZ$, $s>0$, and $\gcd(r,s)=1$. These denominator powers follow by comparing valuations in the integral Weierstrass equation. We obtain $q^2=r(r^2+Bs^4)$. If a rational prime divides $r$ but not $B$, the second factor is a unit at that prime, so its exponent in $r$ must be even. Thus the square class is represented by a signed squarefree divisor of $B$.

Apply this to $E_3$ and its $2$-isogenous curve $E_3':y^2=x^3+36x$, and denote their descent maps to $\QQ^\times/\QQ^{\times2}$ by $\alpha$ and $\alpha'$, respectively. The two images are contained, respectively, in
\[
      \{1,-1,3,-3\}\quad\text{and}\quad\{1,2,3,6\},
\]
where the second list has only positive classes because every nonzero abscissa on $E_3'$ is positive. On $E_3$, all four classes occur at $O,(0,0),(3,0),(-3,0)$, so $\#\operatorname{im}\alpha=4$.

A class $\delta\in\{2,3,6\}$ in the second image would give $x=\delta(U/V)^2$ with coprime integers $U,V$ and $V>0$. Substitution into the curve equation gives
\[
      W^2=\delta U^4+(36/\delta)V^4
\]
for a rational $W$, which must be integral because its square is integral. For $\delta=3$ or $6$, first $3\mid W$. Writing $W=3W_1$ and dividing by $3$ gives, respectively,
\[
      3W_1^2=U^4+4V^4,\qquad
      3W_1^2=2U^4+2V^4.
\]
Modulo $3$, either equation forces $3\mid U,V$, contrary to coprimality. For $\delta=2$, reduction modulo $3$ first gives $3\mid U,W$; writing $U=3U_1,W=3W_1$ gives $W_1^2=18U_1^4+2V^4$, which forces $3\mid V$. Thus $\#\operatorname{im}\alpha'=1$.

Let $\phi:E_3\to E_3'$ be the two-isogeny and $\hat\phi$ its dual. Since $\phi(3,0)=(0,0)$, we have $\phi(E_3(\QQ)[2])=E_3'(\QQ)[\hat\phi]$. The kernel term $E_3'(\QQ)[\hat\phi]/\phi(E_3(\QQ)[2])$ in the descent exact sequence is therefore trivial. The descent formula \cite[X.4, Remark~4.7; X.6]{Silverman2009} gives
\[
      2^{\rank E_3(\QQ)}=\frac{\#\operatorname{im}\alpha\,\#\operatorname{im}\alpha'}{\#E_3(\QQ)[2]}=\frac{4\cdot1}{4}=1,
\]
so $E_3(\QQ)$ has rank zero.

\begin{proof}[Proof of Theorem~\ref{thm:main-fixed}]
      The rational isomorphism
      \[
            E_{3d^2}\longrightarrow E_3,\qquad
            (x,y)\longmapsto(x/d^2,y/d^3)
      \]
      has inverse $(X,Y)\mapsto(d^2X,d^3Y)$, which sends the prescribed point on $E_3$ to a point on $E_{3d^2}$ with abscissa $\theta_t-1$. Since the isomorphism is rational, $E_{3d^2}(\QQ)$ has rank zero and the trace of such a point is two-torsion.

      Theorem~\ref{thm:obstruction} forces the trace to be zero, and Proposition~\ref{thm:trace-pell} then gives $3d^2\in\ZZ$. Write $d=a/b$ in lowest positive terms. Integrality gives $b^2\mid3a^2$, and coprimality gives $b^2\mid3$. Thus $b=1$, so every admissible rational scale is integral.

      Theorem~\ref{thm:main-classification} now gives $v^2=9(2d^4-1)$. Since $3\mid v$, write $v=3B$ with $B\in\ZZ_{>0}$; then
      \[
            B^2+1=2d^4.
      \]
      Ljunggren's theorem \cite{Ljunggren1942}, in the form recorded by Steiner and Tzanakis \cite[Theorem~1]{SteinerTzanakis1991}, gives exactly $(B,d)=(1,1)$ and $(239,13)$.

      For $d=1$, the slopes are $u=3\pm3=0,6$, so $t=u^2+3$ gives $t=3,39$. For $d=13$, we have $n=507$ and $v=717$, giving $u=507\pm717=-210,1224$ and hence $t=44103,1498179$. The ordinate on $E_{3d^2}$ is $u\theta_t+3d^2$; division by $d^3$ gives the ordinates on $E_3$ listed in Table~\ref{tab:four}.

      Conversely, these values satisfy \eqref{eq:identity}, so the table gives points on the curve. Their conjugates generate a rank-two subgroup by Proposition~\ref{prop:rank}, and the rational scaling preserves both infinite order and the rank bound.
\end{proof}

\begin{table}[ht]
      \centering
      \begin{tabular}{rrll}
            \toprule
            \(d\) & \(t\)   & \(X\)                & one choice of \(Y\)         \\
            \midrule
            1     & 3       & \(\theta_t-1\)       & \(3\)                       \\
            1     & 39      & \(\theta_t-1\)       & \(6\theta_t+3\)             \\
            13    & 44103   & \((\theta_t-1)/169\) & \((-210\theta_t+507)/2197\) \\
            13    & 1498179 & \((\theta_t-1)/169\) & \((1224\theta_t+507)/2197\) \\
            \bottomrule
      \end{tabular}
      \caption{Points on \(E_3\) with abscissa \((\theta_t-1)/d^2\). Each row also has the point with ordinate \(-Y\).}
      \label{tab:four}
\end{table}

The four defining parameters are distinct. To distinguish their fields as well, we compare ramified primes.

\begin{proposition}\label{prop:four-fields}
      The four fields in Table~\ref{tab:four} are pairwise nonisomorphic.
\end{proposition}
\nopagebreak
\begin{proof}
      Recall that $\Delta_t=t^2+3t+9$. We first show that a prime $p\ne3$ is totally ramified in $K_t$ whenever $v_p(\Delta_t)$ is not divisible by three. Put $\eta=3\theta_t-t$. The defining equation of $\theta_t$ gives
      \[
            \eta^3=\Delta_t(3\eta+2t+3),\qquad
            (2t+3)^2=4\Delta_t-27.
      \]
      Let $\mathfrak p$ lie above $p$, with ramification index $e_{\mathfrak p}$, and normalize $v_{\mathfrak p}$ to take integer values. Since $p\mid\Delta_t$, the first identity gives $\eta\equiv0\pmod{\mathfrak p}$. The second identity and $p\ne3$ show that $2t+3$ is a unit at $\mathfrak p$. Thus $3\eta+2t+3$ is also a unit, and taking valuations gives
      \[
            3v_{\mathfrak p}(\eta)=e_{\mathfrak p}v_p(\Delta_t).
      \]
      If $3\nmid v_p(\Delta_t)$, then $3\mid e_{\mathfrak p}$. Since the field has degree three, $e_{\mathfrak p}=3$, proving total ramification. A prime not dividing $\Delta_t$ is unramified, since the field discriminant divides $\operatorname{disc}(g_t)=\Delta_t^2$.

      Here
      \[
            \begin{aligned}
                  \Delta_3         & =27,                       & \Delta_{39} & =27\cdot61, \\
                  \Delta_{44103}   & =27\cdot72044701,          &
                  \Delta_{1498179} & =27\cdot61\cdot1362808021.
            \end{aligned}
      \]
      The Euclidean algorithm gives $\gcd(72044701,1362808021)=1$. Their residues modulo $61$ are $41$ and $6$, respectively, and both are $1$ modulo $3$. Thus neither large factor shares a prime divisor with the other or with $3\cdot61$. Neither is a cube:
      \[
            416^3<72044701<417^3,\qquad
            1108^3<1362808021<1109^3.
      \]
      Each therefore has a prime divisor whose exponent is not divisible by three; otherwise it would be a perfect cube. Such a prime ramifies in its corresponding field and in none of the other three. This distinguishes $K_{44103}$ and $K_{1498179}$ from each other and from the remaining fields. Finally, $61$ ramifies in $K_{39}$ and not in $K_3$, which distinguishes the remaining pair.
\end{proof}

\section{The Pell family}\label{sec:pell-family}

The preceding section fixed the curve $E_3$. We now let $n$ vary and describe all points with abscissa $\theta_t-1$ and trace zero from Proposition~\ref{thm:trace-pell}. The equation in that proposition reduces to $b^2-8a^2=1$, so its solutions can be arranged in a single recurrence. For each defining parameter $t_k$, the two choices of sign give two curve parameters; these curve parameters coincide when $k=0$.

\begin{proposition}\label{prop:pell-family}
      Define integers $a_k\geq0,b_k>0$ by
      \[
            b_k+a_k\sqrt8=(3+\sqrt8)^k,\qquad k=0,1,2,\ldots,
      \]
      and put
      \[
            t_k=36a_k^2+3,\qquad n_k^\pm=3(b_k\pm2a_k).
      \]
      Then $n_k^\pm>0$, and the points
      \[
            (\theta_{t_k}-1,6a_k\theta_{t_k}+n_k^-)\in E_{n_k^-}(K_{t_k}),
      \]
      \[
            (\theta_{t_k}-1,-6a_k\theta_{t_k}+n_k^+)\in E_{n_k^+}(K_{t_k})
      \]
      have trace zero. Up to negation, these are all the points in \eqref{eq:pell-data}. Moreover, $n_{k+1}^-=n_k^+$.
\end{proposition}
\begin{proof}
      Put $v=u-\epsilon n\in\ZZ$. The equation $v^2-2n^2=-9$ modulo $3$ forces both $n$ and $v$ to be divisible by $3$. Write $n=3m,v=3A$; then $A^2-2m^2=-1$. If $m$ were even, reduction modulo $8$ would give $A^2\equiv7\pmod8$. Hence $m$ is odd, and the equation modulo $2$ makes $A$ odd. Since $u=3(A+\epsilon m)$, it follows that $6\mid u$.

      After negation, take $u=6a\geq0$. The equation in \eqref{eq:pell-data} becomes
      \[
            (n+\epsilon u)^2=2u^2+9=9(8a^2+1).
      \]
      Therefore $b=|n+\epsilon u|/3$ is a positive integer satisfying $b^2-8a^2=1$. In fact $n+\epsilon u>0$. This is immediate for $\epsilon=1$; for $\epsilon=-1$, a negative value would give $n=6a-3b<0$, since $b>2a$. Thus $n=3b-\epsilon6a$.

      It remains to show that the recurrence gives every solution with $a\geq0,b>0$. Multiplication by $3+\sqrt8$ gives the forward step
      \[
            (b,a)\longmapsto(3b+8a,b+3a).
      \]
      For a solution with $a\geq1$, we have
      \[
            2a<b\leq3a,\qquad 3b>8a.
      \]
      These follow from $b^2=8a^2+1$. The inverse step $(b',a')=(3b-8a,3a-b)$ therefore satisfies $b'>0$ and $0\leq a'<a$, while $b'^2-8a'^2=1$. Repeating the inverse step decreases the nonnegative integer $a$ until it reaches zero, where $b=1$. This proves that every solution comes from $(1,0)$ by the forward recurrence.

      For $\epsilon=1$, the formula $n=3b-\epsilon6a$ gives the first point in the statement. For $\epsilon=-1$, it gives $n=n_k^+$ and ordinate $6a_k\theta_{t_k}-n_k^+$; negation gives the second displayed point. Finally,
      \[
            b_{k+1}-2a_{k+1}
            =(3b_k+8a_k)-2(b_k+3a_k)=b_k+2a_k,
      \]
      so $n_{k+1}^-=n_k^+$.
\end{proof}

\begin{samepage}
      The first values illustrate the two branches and their overlap.

      \begin{center}
            \begin{tabular}{rrrrr}
                  \toprule
                  \(k\) & \(a_k\) & \(t_k\) & \(n_k^-\) & \(n_k^+\) \\
                  \midrule
                  0     & 0       & 3       & 3         & 3         \\
                  1     & 1       & 39      & 3         & 15        \\
                  2     & 6       & 1299    & 15        & 87        \\
                  3     & 35      & 44103   & 87        & 507       \\
                  4     & 204     & 1498179 & 507       & 2955      \\
                  \bottomrule
            \end{tabular}
      \end{center}
\end{samepage}

The sequence $n_k^+$ is strictly increasing, since $n_{k+1}^+-n_k^+=12b_k+36a_k>0$ for every $k\geq0$. Together with $n_{k+1}^-=n_k^+$ and $n_0^-=n_0^+=3$, this shows that, for each curve parameter in this family, exactly two defining parameters $t$ admit a point of trace zero with abscissa $\theta_t-1$. For example, $n=3$ occurs at $t=3,39$, and $n=15$ occurs at $t=39,1299$. The recurrence gives infinitely many pairs of curves and fields as the curve varies.

The fields $K_{t_k}$ are pairwise nonisomorphic. Indeed, Okazaki's classification of isomorphic simplest cubic fields, as stated by Hoshi \cite[Theorem~1.4]{Hoshi2011}, places all nontrivial coincidences for $t\geq-1$ within
\[
      \{-1,5,12,1259\},\quad\{0,3,54\},\quad
      \{1,66\},\quad\{2,2389\}.
\]
The parameters $t_k$ are strictly increasing and are $3$ modulo $36$, whereas only $3$ in this list has that residue. Thus the cited classification rules out isomorphic fields at distinct indices.

\section{Norms and nontorsion traces}\label{sec:quadratic-ordinate}

Let $n$ be a positive integer, with no assumption on $\rank E_n(\QQ)$. The preceding section classifies the prescribed points with trace zero. Real embeddings and norms give necessary conditions for all prescribed points, proving finiteness for fixed $n$ and the exclusions on $E_5$ and $E_6$.

\begin{lemma}\label{lem:real-sign-bound}
      If there exists a point in $E_n(K_t)$ with abscissa $\theta_t-1$ for some integer $n>0$, then
      \[
            2+\frac1{\theta_t}<n<\theta_t-1,
            \qquad 3\le n\le t.
      \]
      In particular, no point with this abscissa exists at $t=-1,0,1,2$.
\end{lemma}
\begin{proof}
      The unique positive root $\theta=\theta_t$ is greater than one, since $g_t(1)=-2t-3<0$. The conjugate abscissae are
      \[
            \theta-1,\qquad -1-\frac1{\theta+1},\qquad -2-\frac1\theta.
      \]
      Each has degree three, so none is $0$ or $\pm n$. The corresponding ordinate squares are therefore strictly positive. At the positive abscissa, $x(x^2-n^2)>0$ gives $n<\theta-1$; at the most negative abscissa it gives $n>2+1/\theta$.

      For $t\ge0$, the identities $g_t(t+1)=-2t-3<0$ and $g_t(t+2)=t^2+3t+1>0$ imply $\theta<t+2$. Integrality now gives $3\le n\le t$. For $t=-1$, $g_{-1}(1)<0<g_{-1}(2)$ gives $\theta<2$, and the upper inequality would force $n<1$.
\end{proof}

The norm of a square in $K_t$ is a rational square. Taking norms gives a necessary equation in $n$ and $t$, without knowing the trace. For $z=\theta_t-1$, the three factors $z,z-n,z+n$ have norms obtained by evaluating $q_t$ at $0,n,-n$. Define
\[
      \begin{aligned}
            \Psi_n(T)={} & (2T+3)\bigl((n+1)(n+2)T-n^3-3n^2+3\bigr)\bigl((n-1)(n-2)T+n^3-3n^2+3\bigr).
      \end{aligned}
\]

\begin{proposition}\label{cor:norm}
      A point in $E_n(K_t)$ with abscissa $\theta_t-1$ gives an integer solution of $W^2=\Psi_n(t)$. For $n\geq4$, the polynomial $\Psi_n$ is a cubic with distinct roots.
\end{proposition}
\begin{proof}
      For any rational $r$, the product of the three conjugates of $z-r$ is $-q_t(r)$ because $q_t$ is monic of degree three. Taking norms in $y^2=z(z-n)(z+n)$ consequently gives
      \[
            \Norm_{K_t/\QQ}(y)^2=-q_t(0)q_t(n)q_t(-n)=\Psi_n(t).
      \]
      Thus $W=\Norm_{K_t/\QQ}(y)$ has integral square and is itself an integer. The leading coefficient of $\Psi_n$ is $2(n^2-1)(n^2-4)$. For a product of three linear polynomials, the discriminant is the product of the squares of their pairwise resultants. Applying this formula gives
      \[
            \operatorname{disc}_T(\Psi_n)
            =4n^6(4n^2-9)^2(n^2-9)^2(n^4-7n^2+9)^2.
      \]
      Both the leading coefficient and the discriminant are nonzero for $n\geq4$; in particular, $n^4-7n^2+9=n^2(n^2-7)+9>0$ in this range. Hence $\Psi_n$ is a cubic with distinct roots.
\end{proof}

For $n\geq4$, the projection from the smooth projective model of the norm equation to the projective $T$-line is a double cover. Over $\overline{\QQ}$, it is branched at the three distinct roots of $\Psi_n$ and at infinity. The Riemann--Hurwitz formula \cite[II.5, Theorem~5.9]{Silverman2009} therefore gives genus one. We can apply a general finiteness theorem for integral values of functions on such curves.

\begin{proof}[Proof of the finiteness assertion in Theorem~\ref{thm:main-positive-rank}]
      Fix $n\geq4$. On the smooth projective model of $W^2=\Psi_n(T)$, the coordinate $T$ is a nonconstant rational function. In Siegel's theorem \cite[IX.3, Corollary~3.2.2]{Silverman2009}, take the base field to be $\QQ$ and $S$ to consist of its archimedean place, so that the ring of $S$-integers is $R_S=\ZZ$. The theorem gives finitely many rational points at which $T$ is integral. Every point with abscissa $\theta_t-1$ gives such a value $T=t$, so only finitely many defining parameters occur.

      It remains to consider $n=1,2,3$. Lemma~\ref{lem:real-sign-bound} excludes $n=1,2$. For $n=3$, the rational-rank calculation in Section~\ref{sec:scaled} allows us to apply Theorem~\ref{thm:main-classification}, which gives exactly $t=3,39$.
\end{proof}

The norm condition is not sufficient. For $n=3,t=5$, the three factors are $13,49,13$, whose product is $91^2$, but the rank-zero classification excludes this parameter.

For each fixed pair $(n,t)$, a quartic equation supplements the norm condition to give a necessary and sufficient criterion. We use the square-root method of Delone and Faddeev \cite[\S~11]{DeloneFaddeev1940}, specialized to the element whose square root is the desired ordinate.

\begin{proposition}\label{prop:square-criterion}
      Let $n$ be a positive integer, and put $z=\theta_t-1$ and $\beta=z^3-n^2z$. Define
      \[
            \begin{aligned}
                  A={} & t^3+(6-n^2)t+3n^2-18,                               \\
                  B={} & 3(n^2-3)t^3+(2n^2-15)t^2-3(n^2-3)(n^2-5)t-9(n^2-3), \\
                  C={} & \Psi_n(t).
            \end{aligned}
      \]
      If $(n,t)\ne(3,3)$, there exists a point in $E_n(K_t)$ with abscissa $z$ if and only if $C=m^2$ for a positive integer $m$ and the polynomial
      \begin{equation}\label{eq:square-quartic}
            H(S)=(S^2-A)^2-8mS-4B
      \end{equation}
      has an integer root $s$. If this is the case and, for any such root, we put $q=(s^2-A)/2$, then the two ordinates are
      \begin{equation}\label{eq:ordinate-recovery}
            y=\pm\frac{s\beta+m}{\beta+q}.
      \end{equation}
      At $(n,t)=(3,3)$, the two ordinates are $\pm3$.
\end{proposition}
\begin{proof}
      Using $q_t(z)=0$ gives
      \[
            \beta=(t-3)z^2+(3t-n^2)z+(2t+3).
      \]
      Since $1,z,z^2$ are linearly independent over $\QQ$, the element $\beta$ is rational exactly when $t=3$ and $n^2=9$. This gives the stated exception, where $\beta=9$. Also $\beta\ne0$ for every pair, since $z$ has degree three and is not $0$ or $\pm n$.

      In the remaining cases, let $\beta_1,\beta_2,\beta_3$ be the conjugates of $\beta$. Expanding their elementary symmetric functions gives
      \[
            A=\Tr_{K_t/\QQ}(\beta),\qquad
            B=\beta_1\beta_2+\beta_1\beta_3+\beta_2\beta_3,\qquad
            C=\Norm_{K_t/\QQ}(\beta).
      \]
      Thus the minimal polynomial of $\beta$ is $X^3-AX^2+BX-C$.

      Suppose $y^2=\beta$. The element $\beta$ is an algebraic integer, so $y$ is also an algebraic integer. Write $y_1,y_2,y_3$ for its conjugates, and put
      \[
            s=y_1+y_2+y_3,\qquad
            q=y_1y_2+y_1y_3+y_2y_3,\qquad m=y_1y_2y_3.
      \]
      These are integers; here $s=\Tr_{K_t/\QQ}(y)$ and $m=\Norm_{K_t/\QQ}(y)$ are the field trace and norm. Since $y\ne0$ and the extension has odd degree, negating $y$ if necessary makes $m>0$. The identities
      \[
            A=s^2-2q,\qquad B=q^2-2sm,\qquad C=m^2
      \]
      then give $H(s)=0$.

      Conversely, suppose $C=m^2$ with $m>0$ and $H(s)=0$ for an integer $s$. Put $q=(s^2-A)/2\in\QQ$. Then $A=s^2-2q$ and $B=q^2-2sm$, so
      \[
            (s\beta+m)^2-\beta(\beta+q)^2
            =-\bigl(\beta^3-A\beta^2+B\beta-C\bigr)=0.
      \]
      The denominator $\beta+q$ is nonzero because $\beta\notin\QQ$. Division proves \eqref{eq:ordinate-recovery}; the two signs give all ordinates because $\beta\ne0$.
\end{proof}

The criterion gives a finite test using integer arithmetic for each fixed pair $(n,t)$. After treating $(n,t)=(3,3)$ separately and checking whether $C$ is a positive square, one tests the integer roots of the monic polynomial $H$. If $A^2-4B\ne0$, every such root divides $A^2-4B$. If $A^2-4B=0$, first take the root $s=0$; the other roots are roots of the monic cubic $H(S)/S$, whose constant term is $-8m\ne0$. Thus a finite list of signed divisors suffices in either case.

Common divisors of the three norm factors restrict their square classes. For $n=5$, these restrictions force a congruence on $t$ that is incompatible with an ordinate square at the prime $5$.

\begin{proof}[Proof of the $E_5$ assertion in Theorem~\ref{thm:main-positive-rank}]
      Suppose such a point exists. Lemma~\ref{lem:real-sign-bound} gives $t\ge5$. By Proposition~\ref{cor:norm}, the product of the positive odd integers
      \[
            L_1=2t+3,\qquad L_2=42t-197,\qquad L_3=12t+53
      \]
      is a square. Write $L_1=\kappa_1\rho_1^2$, $L_2=\kappa_2\rho_2^2$, and $L_3=\kappa_3\rho_3^2$, with positive squarefree $\kappa_1,\kappa_2,\kappa_3$ and positive odd integers $\rho_1,\rho_2,\rho_3$. A prime occurring in an odd number of these three squarefree parts would have odd valuation in $L_1L_2L_3$, so each prime occurs in either zero or two of them.

      The identities
      \[
            21L_1-L_2=260,\qquad 6L_1-L_3=-35,\qquad
            2L_2-7L_3=-765
      \]
      show that a common prime of $L_1,L_2$ lies in $\{2,5,13\}$, one of $L_1,L_3$ lies in $\{5,7\}$, and one of $L_2,L_3$ lies in $\{3,5,17\}$.

      The three factors are odd, so $2$ divides none of their squarefree parts. Also $L_2\equiv1$ and $L_3\equiv2\pmod3$, so $3$ divides neither $\kappa_2$ nor $\kappa_3$; the parity condition on the squarefree parts excludes it from $\kappa_1$ as well. Consequently
      \[
            \kappa_1=5^{e_{5,1}}7^{e_7}13^{e_{13}},\qquad \kappa_2=5^{e_{5,2}}13^{e_{13}}17^{e_{17}},\qquad
            \kappa_3=5^{e_{5,3}}7^{e_7}17^{e_{17}},
      \]
      where $e_7,e_{13},e_{17},e_{5,1},e_{5,2},e_{5,3}\in\{0,1\}$ and $e_{5,1}+e_{5,2}+e_{5,3}$ is even.

      Odd squares are $1$ modulo $8$. Since $L_1\equiv L_2\pmod8$, we have $\kappa_1\equiv \kappa_2\pmod8$. The odd square $\kappa_1\kappa_2\kappa_3$ then gives $\kappa_3\equiv1\pmod8$. But $\kappa_3\equiv5^{e_{5,3}}7^{e_7}\pmod8$, and among the four choices only $e_{5,3}=e_7=0$ gives one. Thus $e_{5,1}=e_{5,2}$ and $\kappa_3=17^{e_{17}}$.

      Reduction of $L_3=\kappa_3\rho_3^2$ modulo $3$ forces $e_{17}=1$, because $L_3\equiv2$ and $3\nmid \rho_3$. Reduction of $L_2=5^{e_{5,1}}13^{e_{13}}17\rho_2^2$ modulo $3$ then forces $e_{5,1}=1$. In particular $5\mid \kappa_1$, so $t\equiv1\pmod5$.

      For every such $t$, the polynomial $g_t$ has the simple root $3$ modulo $5$:
      \[
            g_t(3)=17-12t\equiv0,\qquad g_t'(3)=24-7t\equiv2\pmod5.
      \]
      Hensel's lemma \cite[IV.1, Lemma~1.2]{Silverman2009} lifts this simple root to a root in $\ZZ_5$. Since $g_t$ is the minimal polynomial of $\theta_t$, sending $\theta_t$ to the lifted root defines an embedding $K_t\hookrightarrow\QQ_5$. Under this embedding the proposed ordinate square is a unit congruent to
      \[
            (3-1)^3-25(3-1)\equiv3\pmod5,
      \]
      which is not a square in $\FF_5$. A square root in $\QQ_5$ would be a unit whose reduction has square $3$ in $\FF_5$. This contradiction proves the theorem.
\end{proof}

The rational point $(45,300)\in E_5(\QQ)$ is nontorsion by \eqref{eq:torsion}, so the rank-zero classification does not apply. The norm equation itself has the solution $(t,W)=(31,5525)$; the norm condition does not exclude this parameter.

For $E_6$, the norm factors already give a contradiction. Thus the local argument at $5$ used for $E_5$ is unnecessary in this case.

\begin{proof}[Proof of the $E_6$ assertion in Theorem~\ref{thm:main-positive-rank}]
      A point $(\theta_t-1,y)\in E_6(K_t)$ would give $t\geq6$ by Lemma~\ref{lem:real-sign-bound}, and Proposition~\ref{cor:norm} would make the product of the positive odd integers
      \[
            L_1=2t+3,\qquad L_2=56t-321,\qquad L_3=20t+111
      \]
      a square. Modulo $8$, their product is $6t+3$. An odd square is $1$ modulo $8$, so $t\equiv1\pmod4$ and $L_3\equiv3\pmod8$.

      Write $L_3=\kappa\rho^2$, where $\kappa$ is positive and squarefree and $\rho$ is a positive odd integer. Any prime dividing $\kappa$ occurs to odd order in $L_3$, so it must also divide $L_1$ or $L_2$. The identities
      \[
            10L_1-L_3=-81,\qquad
            5L_2-14L_3=-3159=-3^5\cdot13
      \]
      therefore give $\kappa\in\{1,3,13,39\}$. Since $\rho^2\equiv1\pmod8$ and $L_3\equiv3\pmod8$, only $\kappa=3$ is possible. But $20t+111=3\rho^2$ then gives $\rho^2\equiv2\pmod5$, a contradiction.
\end{proof}

An ordinate of degree at most one in $z=\theta_t-1$ places the conjugate points on a rational line and gives trace zero. An ordinate represented by a quadratic polynomial in $z$ would instead place them on a parabola whose fourth intersection with the curve is rational and determines the trace.

\begin{samepage}
      \begin{proposition}\label{prop:quadratic-trace}
            Let $P=(z,y)\in E_n(K_t)$ with $z=\theta_t-1$, and write its ordinate uniquely as $y=J(z)$, with $J(Z)=j_2Z^2+j_1Z+j_0\in\QQ[Z]$. Then $\Tr(P)=O$ if and only if $j_2=0$. If $j_2\ne0$, then
            \[
                  J(Z)^2-f_n(Z)=j_2^2q_t(Z)(Z-r_4)
            \] with $r_4=3-t-\frac{2j_2j_1-1}{j_2^2}$ and the point $R=(r_4,J(r_4))$ is rational and nontorsion with $\Tr(P)=-R$.
      \end{proposition}
\end{samepage}
\begin{proof}
      The basis $1,z,z^2$ gives the unique polynomial $J$. Lemma~\ref{lem:line} shows that trace zero is equivalent to an expression of degree at most one in $z$ for the ordinate, hence to $j_2=0$.

      Suppose $j_2\ne0$. The quartic $J^2-f_n$ vanishes at $z$, so it is divisible by $q_t$. Its leading coefficient is $j_2^2$, and the remaining factor has degree one over $\QQ$. We may therefore write it as $j_2^2q_t(Z)(Z-r_4)$ for a rational $r_4$. Comparing the cubic coefficients gives the displayed value of $r_4$, and evaluating at $r_4$ gives $J(r_4)^2=f_n(r_4)$. Thus $R=(r_4,J(r_4))$ is a rational point.

      To recover the trace, consider the function $y-J(x)$ on $E_n$. At $O$, the coordinates $x$ and $y$ have pole orders two and three. Since $j_2\ne0$, the term $j_2x^2$ gives $y-J(x)$ a pole of order four, with no other poles. Its zeros include the three conjugates of $P$ and the fourth point $R$. Their abscissae are distinct because $q_t$ is separable and has no rational root. These four zeros account for the full pole order, so the divisor is
      \[
            (P)+(\sigma P)+(\sigma^2P)+(R)-4(O).
      \]
      The principal-divisor criterion for the elliptic group law \cite[III.3, Corollary~3.5]{Silverman2009} gives $\Tr(P)+R=O$. Since $j_2\ne0$, the trace is nonzero. Theorem~\ref{thm:obstruction} excludes nonzero two-torsion, and \eqref{eq:torsion} therefore makes $R$ nontorsion.
\end{proof}

We do not know whether every point in $E_n(K_t)$ with abscissa $\theta_t-1$ has trace zero when $\rank E_n(\QQ)>0$. A counterexample would have an ordinate represented by a polynomial of degree two in $z$ and a nontorsion rational trace, and would give an integral solution of the norm equation. For fixed $n$, the proof of Theorem~\ref{thm:main-positive-rank} gives finiteness but no effective bound on $t$. Proposition~\ref{prop:square-criterion} decides existence for each fixed pair $(n,t)$, but does not give such a bound or settle the trace-zero question.

\section{Integral shifts and a fixed-curve family}\label{sec:shifted}

The restriction to $\theta_t-1$ in Theorem~\ref{thm:main-fixed} is essential. We now replace it by $\theta_t-r$, where $r$ is an integer. The rational-line argument still describes trace-zero points, but its coefficient equations change.

\begin{proposition}\label{prop:shifted-criterion}
      Fix integers $r$, $n>0$, and $t\geq-1$. A point $(\theta_t-r,y)\in E_n(K_t)$ has trace zero if and only if there are integers $c,w$ such that
      \begin{equation}\label{eq:shifted-squares}
            \begin{aligned}
                  w^2 & =rn^2+1-r^3,         \\
                  c^2 & =(r+1)n^2-(r+1)^3-1, \\
                  t   & =(c+w)^2+3r,         \\
                  y   & =(c+w)\theta_t+w.
            \end{aligned}
      \end{equation}
      Moreover, for each fixed integer $r\geq2$, only finitely many integer quadruples $(n,t,c,w)$ satisfy \eqref{eq:shifted-squares}.
\end{proposition}
\begin{proof}
      The conjugate abscissae are distinct, so Lemma~\ref{lem:line} makes trace zero equivalent to $y=u\theta_t+w$ with $u,w\in\QQ$. Comparing the two monic cubics in
      \[
            (X-r)^3-n^2(X-r)-(uX+w)^2=g_t(X)
      \]
      gives
      \[
            t=u^2+3r,\qquad n^2+2uw=u^2+3(r^2+r+1),
            \qquad w^2=rn^2+1-r^3.
      \]
      Thus $u^2,w^2\in\ZZ$, which forces $u,w\in\ZZ$. On setting $c=u-w$, the middle identity becomes the second equation in \eqref{eq:shifted-squares}. Conversely, those equations recover all three coefficient identities. They therefore give a point whose conjugates lie on the rational line $y=u(x+r)+w$, proving trace zero.

      Fix $r\geq2$. Every solution maps to an integral point $(N,Y)=(n,wc)$ on
      \begin{equation}\label{eq:shifted-quartic}
            Y^2=(rN^2+1-r^3)\bigl((r+1)N^2-(r+1)^3-1\bigr).
      \end{equation}
      Both quadratic factors have two distinct roots, and their resultant is
      \[
            \bigl((2r+1)(r^2+r+1)\bigr)^2\ne0.
      \]
      Hence the quartic has four distinct roots. Over $\overline{\QQ}$, projection to the projective $N$-line is a double cover branched at these roots, so the Riemann--Hurwitz formula \cite[II.5, Theorem~5.9]{Silverman2009} gives genus one for its smooth projective model. With the same choice of base field $\QQ$ and ring $R_S=\ZZ$ as above, Siegel's theorem \cite[IX.3, Corollary~3.2.2]{Silverman2009}, applied to the nonconstant function $N$, gives finitely many rational points with integral $N$. For each $n$, there are at most four sign choices for $c,w$, and these determine $t$ and $y$. This proves finiteness.
\end{proof}

At $r=1$, the conditions are $w^2=n^2$ and the Pell equation in Proposition~\ref{thm:trace-pell}; the quartic is singular. For $r\geq2$, finiteness does not give an effective classification, and a point on \eqref{eq:shifted-quartic} need not make each factor a square. The criterion concerns trace-zero points: Theorem~\ref{thm:obstruction} excludes nonzero two-torsion trace only for $r=1$ in this paper.

\begin{example}\label{ex:shift73}
      For $r=73$ and $n=75$, the two square equations give $w^2=147^2$ and $c^2=105^2$. Taking $w=147$ and $c=-105,105$ gives, respectively,
      \[
            (u,t)=(42,1983),\qquad (252,63723).
      \]
      Thus $(\theta_t-73,u\theta_t+147)$ is a trace-zero point on $E_{75}(K_t)$ for either pair. Scaling gives
      \[
            \left(\frac{\theta_t-73}{25},\frac{u\theta_t+147}{125}\right)\in E_3(K_t).
      \]
      These abscissae are not of the form $(\theta_t-1)/d^2$ for a rational $d>0$ with this same chosen generator: comparing coefficients of $\theta_t$ would give $d^2=25$, whereas the constant terms disagree. No completeness assertion for shift $73$ is made here.
\end{example}

For $n=r+2$, the two square conditions become $w^2=(2r+1)^2$ and $c^2=(2r+1)(r+2)$. Varying both $r$ and $d$ and imposing $n=3d^2$ keeps the scaled curve equal to $E_3$ and leads to a negative Pell equation.

\begin{theorem}\label{thm:shifted-E3}
      Define positive integers $A_k,d_k$ by
      \[
            A_k+d_k\sqrt2=(1+\sqrt2)(3+2\sqrt2)^k,
            \qquad k=0,1,2,\ldots,
      \]
      and set
      \[
            r_k=3d_k^2-2,\qquad u_k=3A_k(A_k-d_k),\qquad
            w_k=3A_k^2,\qquad t_k^*=u_k^2+9d_k^2-6.
      \]
      Then
      \[
            P_k=\left(\frac{\theta_{t_k^*}-r_k}{d_k^2},
            \frac{u_k\theta_{t_k^*}+w_k}{d_k^3}\right)\in E_3(K_{t_k^*})
      \]
      has trace zero and infinite order. Its conjugates generate a subgroup of rank two, and the fields $K_{t_k^*}$ are pairwise nonisomorphic. In particular, $\rank E_3(K_{t_k^*})\geq2$ for every $k$.
\end{theorem}
\begin{proof}
      Taking norms in $\QQ(\sqrt2)$ gives $A_k^2-2d_k^2=-1$. Fix $k$ and suppress the subscripts. Put $n=3d^2=r+2$ and $c=-3dA$. Since $2r+1=3A^2$, we have
      \[
            w=2r+1,\qquad c^2=(2r+1)(r+2),\qquad u=c+w.
      \]
      These are exactly \eqref{eq:shifted-squares}, with $t=t_k^*$. Proposition~\ref{prop:shifted-criterion} supplies a trace-zero point on $E_{3d^2}$; scaling by $(x,y)\mapsto(x/d^2,y/d^3)$ gives $P_k$. Equivalently, its conjugates lie on the rational line
      \[
            y=\frac{u}{d}x+\frac{ur+w}{d^3}.
      \]
      The abscissa has degree three, so the torsion theorem \eqref{eq:torsion} makes $P_k$ nontorsion. The independence argument in Proposition~\ref{prop:rank} uses only this fact and the trace-zero relation, so it also gives rank two for the subgroup generated by these conjugates.

      The recurrence is
      \[
            (A_{k+1},d_{k+1})=(3A_k+4d_k,2A_k+3d_k).
      \]
      It preserves positivity and oddness. Also $A_k\geq d_k$, with equality only at $k=0$, and
      \[
            A_{k+1}-d_{k+1}=A_k+d_k>A_k-d_k.
      \]
      Thus $d_k$ and $u_k$ increase strictly, and so does $t_k^*$. Oddness gives $6\mid u_k$ and $d_k^2\equiv1\pmod4$, hence $t_k^*\equiv3\pmod{36}$. The exceptional list in Okazaki's classification, as stated by Hoshi \cite[Theorem~1.4]{Hoshi2011} and displayed in Section~\ref{sec:pell-family}, contains only one parameter congruent to $3$ modulo $36$. Hence the strictly increasing parameters $t_k^*$ define pairwise nonisomorphic fields.
\end{proof}

The first values of $t_k^*$ are $3,1983,2186139,2519293143$. For each $k$, this gives a rank-two subgroup of trace-zero points, without classifying all such points or all shifted points on $E_3$.



\bibliographystyle{alpha}
\bibliography{marked_points_on_congruent_number_curves}
\end{document}